\documentclass[12pt]{amsart}
\usepackage{fullpage,url,amssymb,enumerate}
\usepackage[all]{xy}

\usepackage[OT2,T1]{fontenc}

\usepackage{color}

\newcommand{\Z}{\mathbb{Z}}

\newtheorem{theorem}{Theorem}

\newtheorem{corollary}{Corollary}[theorem]
\newtheorem{proposition}{Proposition}
\newtheorem{example}{Example}

\theoremstyle{definition}

\theoremstyle{remark}

\definecolor{darkgreen}{rgb}{0,0.5,0}
\usepackage[
		hidelinks,
        pdfauthor={Bruce Zimov},
        pdftitle={A lower bound limiting solutions in the hyperbolic case of the generalized Fermat equation},
]{hyperref}

\usepackage[alphabetic,backrefs,lite]{amsrefs} 

\usepackage{comment}

\begin{document}

\title{A lower bound limiting solutions in the hyperbolic case of the generalized Fermat equation}

\subjclass[2020]{Primary 11D41}

\author{Bruce Zimov}
\address{Calimesa Research Institute,
		 33562 Yucaipa Blvd 4-321,
		 Yucaipa, CA 92399,
		 USA	
         }
\email{katcha997@aol.com}

\date{December 5, 2020}

\begin{abstract}
  We find a lower bound for $\displaystyle{\chi=\frac{1}{p}+\frac{1}{q}+\frac{1}{r}}$ limiting any solution in the hyperbolic case of the Generalized Fermat Equation $x^p+y^q=z^r$.
\end{abstract}

\maketitle


\section{Introduction}

 Let $p,q,r,x,y,z \in \Z$, $(x,y,z)=1$, and $x,y,z\geq 2$. The hyperbolic case of the Generalized Fermat Equation is
\begin{center}
	$x^p + y^q = z^r$
\end{center}
with
\begin{center}
	$\displaystyle \frac{1}{p}+\frac{1}{q}+\frac{1}{r}<1$.
\end{center}
The absence of solutions for $p,q,r\geq3$ has been recently surveyed in~\cite{BCDY}.
A lower bound for $\displaystyle \frac{1}{p}+\frac{1}{q}+\frac{1}{r}$ means that there are no solutions to the Generalized Fermat Equation with exponents $(p,q,r)$ below this lower bound.
Assuming Baker's Explicit $abc$-conjecture, theorems have been proven with explicit lower bounds for $\displaystyle \frac{1}{p}+\frac{1}{q}+\frac{1}{r}$.

In particular, Laishram and Shorey ~\cite{LS} showed that
\begin{center}
	$\displaystyle \frac{4}{7}<\frac{1}{p}+\frac{1}{q}+\frac{1}{r}$
\end{center}
Chim Shorey and Sinha ~\cite{CSS} improved this to
\begin{center}
	$\displaystyle \frac{1}{1.72}<\frac{1}{p}+\frac{1}{q}+\frac{1}{r}$
\end{center}
Denote the greatest square-free factor of $xyz$ as $G=G(x,y,z)=\prod\limits_{p|xyz} p$, the radical of $xyz$, where $p$ is a prime. Without relying on any conjecture, Stewart and Yu~\cite{SY} proved the following:
There exists an effectively computable positive constant $c$ such that for all positive integers $x$,$y$, and $z$ with $(x,y,z)=1$, $z>2$, and $x+y=z$,
\begin{center}
$\displaystyle log (z)<G^{\displaystyle 2/3+\frac{c}{log~log(G)}}$
\end{center}
Stewart and Yu~\cite{SYY}\cite{SY3} subsequently improved the upper bound to 
\begin{equation}\label{E:SY}
\displaystyle {log (z)<G^{\displaystyle 1/3+\frac{c}{log~log(G)}}}
\end{equation}
Wong Chi Ho~\cite{H} obtained $c=15$ for the effectively computable constant in the upper bound                                                                                                                                        . Using this result, we will obtain a lower bound for $\displaystyle{\frac{1}{p}+\frac{1}{q}+\frac{1}{r}}$ in the hyperbolic case of the Generalized Fermat Equation without relying on any conjecture.  In this paper, we will prove the following:

\begin{theorem}\label{T:GFE}
  Let $p,q,r,x,y,z \in \Z$.
  
  If $(x,y,z)=1$, $x,y,z\geq 2$, and $p,q,r \geq 3$, 
  then the equation
  \begin{equation} \label{E:GFE}
    x^p + y^q = z^r
  \end{equation}
  has no solutions with
\begin{equation}\label{E:PG}
\displaystyle\frac{3\cdot log~log(z^r)}{log(z^r)}\displaystyle{\frac{1}{\bigg(1+\displaystyle{\frac{45}{log~log(G)}}\bigg)}}< \frac{1}{p}+\frac{1}{q}+\frac{1}{r}
\end{equation}  
where $G=G(x,y,z)$.
\end{theorem}

\begin{center}
\end{center}

\section{Proof of Theorem ~\ref{T:GFE}}
\label{sec:proof}

\begin{proof}
  
  From (\ref{E:GFE}) above,
  \begin{center}
   $x<z^{\displaystyle r/p},y<z^{\displaystyle r/q}$.
  \end{center}
  By the Main Theorem of Wong's thesis~\cite{H}, we have
  \begin{center}
$log (z^{\displaystyle r})<G^{\displaystyle 1/3+\frac{\displaystyle 15}{\displaystyle log~log(G)}}$
\end{center} 
We note that ${\displaystyle G=G(x,y,z)\leq xyz<z^{\displaystyle r\chi}}$ where $\displaystyle{\chi=\frac{1}{p}+\frac{1}{q}+\frac{1}{r}}$.
Then,
  \begin{center}
$log (z^{\displaystyle r})<z^{{\displaystyle r\chi}\bigg(\displaystyle{1/3+\frac{15}{log~log(G)}}\bigg)}$
\end{center} 
  \begin{center}
${\displaystyle\frac{log~log(z^r)}{log(z^r)}<\chi\bigg(1/3+\frac{15}{log~log(G)}\bigg)}$
\end{center}
\begin{center}
$\displaystyle\frac{3\cdot log~log(z^r)}{log(z^r)}\displaystyle{\frac{1}{\bigg(1+\displaystyle{\frac{45}{log~log(G)}}\bigg)}}<\chi$
\end{center}

There is no solution to (\ref{E:GFE}) for $\chi$ below this lower bound. Hence, the assertion. 
\end{proof}

\section{Corollaries}

\subsection{Negative lower bound}

\begin{proposition}
If the lower bound to $\chi$ is negative, then no $(p,q,r)$ are excluded for solving (\ref{E:GFE}) .
\end{proposition}
\begin{proof}
$z$,$r$,$\chi$, and $G$ are positive. There are no $\chi$ values below any negative lower bound. Hence, no $(p,q,r)$ are excluded for solving $(\ref{E:GFE})$ in this case.
\end{proof}
\begin{example}Negative lower bound for $\chi$.
\end{example}
Stewart and Yu~\cite{SY3} further improved the upper bound of (\ref{E:SY}) to  
\begin{center}
$\displaystyle {z<e^{c\cdot G^{1/3}(log(G))^3}}$
\end{center}
Applied to (\ref{E:GFE}), we get
\begin{equation}
\displaystyle {log (z^r)<c\cdot G^{1/3}(log(G))^3}\label{E:SY1}
\end{equation}
We proceed with $G<z^{r\chi}$ as in our proof of Theorem \ref{T:GFE} above.
\begin{center}
$\displaystyle {log (z^r)<c\cdot z^{r\chi/3}(log(z^{r\chi}))^3}$

$\displaystyle {log \bigg(\frac{log (z^r)}{c}\bigg) < \frac{\chi}{3}\cdot log(z^r)+3\cdot log~log(z^{r\chi})}$

$\displaystyle {log~log(z^r)-log(c) < \frac{\chi}{3}\cdot log(z^r)+3\cdot log(\chi)+3\cdot log~log(z^r)}$

$\displaystyle {-(2\cdot log~log(z^r)+log(c))\cdot \frac{3}{log(z^r)} < \chi+9\cdot \frac{log(\chi)}{log(z^r)}}<\chi\cdot \bigg(1+\frac{9}{log(z^r)}\bigg)$

$\displaystyle {\frac{-3\cdot (2\cdot log~log(z^r)+log(c))}{log(z^r)+9} < \chi}$
\end{center}
Clearly the left hand side is negative unless
\begin{center}
$0<r<\displaystyle{\frac{1}{\sqrt{c}~log(z)}}$
\end{center}

Chim Kwok Chi~\cite{C} obtained $c=e^{2.6\times 10^{44}}$ for the effectively computable constant. Thus,
\begin{center}
$0<r<\displaystyle{\frac{1}{\sqrt{c}~log(z)}}<3$
\end{center}
Hence,by the proposition, no $(p,q,r)$ are excluded for solving (\ref{E:GFE}) given (\ref{E:SY1}).
\subsection{Boundedness of G}
\begin{corollary}

Let $\phi(x)=\displaystyle\frac{3\cdot log~log(x)}{log(x)}$ and $\displaystyle{\chi=\frac{1}{p}+\frac{1}{q}+\frac{1}{r}}$.
\begin{equation}\label{E:C}
log~log(G)<\displaystyle\frac{45}{\displaystyle\frac{\phi(z^r)}{\chi}-1}
\end{equation}
\end{corollary}

\begin{proof}
This is a consequence of (\ref{E:PG}).
\end{proof}
Consider the values of $\phi(z^r)$ as given in Table \ref{Tab:m(z)}. Observe that for $z^r>2^4$, $\phi(z^r)$ is a decreasing function. Consequently, $\phi(z^r)$ is at a maximum at a value of $1.1034$.
If $p$ and $q$ go to $\infty$,we have
\begin{center}
$\lim\limits_{p,q\to\infty} \displaystyle\frac{\phi(z^r)}{\chi}=\frac{3\cdot log~log(z^r)}{log(z)}$
\end{center}
For $p=q=3$ at their minimum, we have
\begin{center}
$\displaystyle\frac{\phi(z^r)}{\chi}=\frac{3}{2\cdot r+3}\cdot \frac{3\cdot log~log(z^r)}{log(z)}$
\end{center}

In general, we have
\begin{center}
$\displaystyle\frac{\phi(z^r)}{\chi}=\frac{p\cdot q}{(p+q)\cdot r+p\cdot q}\cdot \frac{3\cdot log~log(z^r)}{log(z)}$
\end{center}

If $\displaystyle\frac{\phi(z^r)}{\chi}>1$ we have an upper bound to $G$.

If $\displaystyle\frac{\phi(z^r)}{\chi}=1$ then $G$ is unbounded.

$\displaystyle\frac{\phi(z^r)}{\chi}<1$ is a contradiction because $G\geq 30=2\cdot3\cdot5$.

\begin{table}[b] \renewcommand{\arraystretch}{1.2}
\[ \begin{array}{|c|c|} \hline
  z^r & \phi(z^r)\hspace{1mm} \\\hline
  2^3 & 1.0562  \\\hline
  2^4 & 1.1034  \\\hline
  3^3 & 1.0856  \\\hline
  2^5 & 1.0759  \\\hline
  2^6 & 1.0281  \\\hline
  3^4 & 1.0106  \\\hline
  5^3 & 0.9783  \\\hline
  2^7 & 0.9765  \\\hline
  \end{array}\]

\caption{First $8$ values of $\phi(z^r)$. Calculations were done with the Java Math Library.}
\label{Tab:m(z)}
\end{table}
\newpage
\section*{Acknowledgement} The author would like to thank Cameron Stewart for his valuable comments on an earlier version of this paper.


\end{document}